\documentclass[a4paper,12pt,leqno]{article}

\usepackage{amsmath,amssymb,amsfonts,amsthm}
\usepackage{hyperref}
\usepackage{xcolor}
\usepackage{OldStandard}

\usepackage[sorting=ynt]{biblatex}
\title{\bfseries On A Divisor Sum Involving Pairwise Relatively Prime Positive Integers}
\author{\itshape Masum Billal}

\newtheorem{theorem}{\textsc{Theorem}}

\newtheorem{lemma}{\textsc{Lemma}}

\theoremstyle{definition}
\newtheorem*{remark}{Remark}

\begin{document}
	\maketitle
		\begin{abstract}
			The number of tuples with positive integers pairwise relatively prime to each other with product at most $n$ is considered. A generalization of $\mu^{2}$ where $\mu$ is the M\"{o}bius function is used to formulate this divisor sum and establish some identities. One such identity is a weighted sum of reciprocal of square-free numbers not exceeding $n$. Some auxiliary number theoretic functions are introduced to formulate this sum.
		\end{abstract}
	The generalized divisor function $\tau_{k}(n)$ and its summation function $D_{k}(n)$ are often of high interest in the literature of number theory.
		%\begin{align}
		%	D_{k}(n)
		%		& = \sum_{i=1}^{n}\tau_{k}(i)\label{eqn:D}
		%\end{align}
	%It can be proven using elementary means (see \textcite[p. $15$]{friedlander_iwaniec_2010}) that
	%	\begin{align*}
	%		\sum_{i=1}^{n}\tau_{k}(i)
	%			& = \dfrac{1}{(k-1)!}x(\log{x})^{k-1}+O\left(x(\log{x})^{k-2}\right)
	%	\end{align*}
	We can write $D_{k}(n)$ as
		\begin{align*}
			D_{k}(n)
				%& = \sum_{i=1}^{n}\tau_{k}(n)\\
				& = \sum_{\substack{a_{1}\cdots a_{k}\leq n\\1\leq a_{1},\ldots,a_{k}\leq n}}1
		\end{align*}
	The purpose of this paper is to look at a similar function
		\begin{align*}
			\mathcal{T}_{k}(n)
				& = \sum_{\substack{a_{1}\cdots a_{k}\leq n\\1\leq a_{1},\ldots, a_{k}\leq n\\\gcd(a_{i},a_{j})=1,i\neq j}}1
		\end{align*}
	In words, $\mathcal{T}_{k}(n)$ is the number of ordered tuples of pairwise relatively prime positive integers $(a_{1},\ldots,a_{k})$ with product at most $n$. Functions of the exact form of $\mathcal{T}$ have been rarely considered in the literature. Although, similar functions have been discussed by various authors. \textcite{lehmer1900} seems to be the first one to study of functions of the form
		\begin{align*}
			\sum_{\substack{1\leq a_{1},\ldots, a_{k}\leq n\\\gcd(a_{1},\ldots,a_{k})=1}}1
		\end{align*}
	\textcite{nymann1972,nymann1975} also considers functions of this form. This kind of functions are generalizations of Euler's Totient function as can be seen in \textcite{shonhiwa1999,toth2002}. However, $\mathcal{T}$ differs from such functions fundamentally due to the additional restriction that $a_{1}\cdots a_{k}\leq n$. The closest consideration to $\mathcal{T}$ seems to be done by \textcite{toth2002} where he looked at the sum without the restriction on product. Although \textcite{toth2002} achieved some nice results, they do not help much with the restriction we are considering. We will not consider the following generalization in this paper.
		\begin{align*}
			\mathcal{T}_{k,r}(n)
				& = \sum_{\substack{a_{1}\cdots a_{k}\leq n\\1\leq a_{1},\ldots,a_{k}\leq n\\\gcd(a_{i},a_{j})=1,i\neq j\\\gcd(a_{i},r)=1}}1
		\end{align*}
	We will establish some identities related to $\mathcal{T}$. $\mathcal{T}$ possesses some probabilistic meaning similar to \textcite{benkoski1976,toth2002,nymann1972} in the sense that the probability that $k$ positive integers not exceeding $n$ are pairwise relatively prime and that their product also does not exceed $n$ is $\mathcal{T}_{k}/n^{k}$. Dealing with $\mathcal{T}$ is not completely trivial. For example, one would expect that the generalized convolution and inversion might be a good start but a little manipulation shows that it does not help us with the sum $\mathcal{T}$. Rather we can find a formula for $\mathcal{T}$ if we write
		\begin{align*}
			\mathcal{T}_{k}(n)
				& = \sum_{\substack{a_{1}\cdots a_{k}\leq n\\1\leq a_{1}\ldots\leq a_{k}\leq n\\\gcd(a_{i},a_{j})=1,i\neq j}}1\\
				& = \sum_{\substack{a_{1}\cdots a_{k}\leq n-1\\1\leq a_{1}\ldots\leq a_{k}\leq n\\\gcd(a_{i},a_{j})=1,i\neq j}}1+\sum_{\substack{a_{1}\cdots a_{k}= n\\1\leq a_{1}\ldots\leq a_{k}\leq n\\\gcd(a_{i},a_{j})=1,i\neq j}}1\\
				& = \sum_{\substack{a_{1}\cdots a_{k}\leq n-1\\1\leq a_{1}\ldots\leq a_{k}\leq n\\\gcd(a_{i},a_{j})=1,i\neq j}}1+\rho_{k}(n)
		\end{align*}
	Here, $\rho_{k}(n)$ is the number of ways to write $n$ as a product of $k$ pairwise relatively prime positive integers. Fortunately, we have the beautiful result\footnote{It may not be well known to everyone, but we can induct on $\omega(n)$ to prove it. So we leave the proof as an exercise to the reader unfamiliar with this beautiful result.} that $\rho_{k}(n)=k^{\omega(n)}$ where $\omega(n)$ is the number of distinct prime divisors of $n$. We have the following.
		\begin{theorem}\label{thm:omegasum}
			\begin{align}
				\mathcal{T}_{k}(n)
					& = \sum_{i=1}^{n}k^{\omega(n)}
			\end{align}
		\end{theorem}
	Although \autoref{thm:omegasum} gives us an exact formula, we will look for an alternative formulation of $\mathcal{T}$.
		\begin{lemma}\label{lem:summation}
			Let $f$ be a multiplicative arithmetic function and the summation function of $f$ be $F(x)=\sum_{n\leq x}f(n)$. Then there exists a computable unique multiplicative arithmetic function $g$ such that
				\begin{align*}
					F(x)
						& = \sum_{n\leq x}\left\lfloor{\dfrac{x}{n}}\right\rfloor g(n)
				\end{align*}
		\end{lemma}

		\begin{proof}
			Let $g$ be an arithmetic function such that
				\begin{align*}
					f(n)
						& = \sum_{d\mid n}g(d)
				\end{align*}
			It is easy to see from this representation that $g$ is unique and $g$ can be computed using the recursion $g(1)=f(1)$ and
				\begin{align*}
					g(n)
						& = f(n)-\sum_{\substack{d\mid n\\d<n}}g(d)
				\end{align*}
			If $\mu$ is the M\"{o}bius function, $g=\mu\ast f$ and since $\mu$ and $f$ both are multiplicative, their convolution is multiplicative as well. Then we get
				\begin{align*}
					F(x)
						& = \sum_{n\leq x}f(n)\\
						& = \sum_{n\leq x}\sum_{d\mid n}g(d)\\
						& = \sum_{n\leq x}\left\lfloor{\dfrac{x}{n}}\right\rfloor g(n)
				\end{align*}
		\end{proof}
	\autoref{lem:summation} allows us to easily calculate $\mathcal{T}$ in a different way. Note that $\omega(n)$ is an additive function, so $\rho_{k}(n)=k^{\omega(n)}$ is a multiplicative function. So, we can apply \autoref{lem:summation} for $\rho$ if we know the value of $g(n)$. In fact, we have the following result.
		\begin{theorem}
			Let $n,k$ be positive integers. Then
				\begin{align*}
					k^{\omega(n)}
						& = \sum_{d\mid n}\delta_{k}(d)
				\end{align*}
			where $\delta$ is a generalization of $\mu^{2}$ defined as
				\begin{align*}
					\delta_{k}(n)
						& =
							\begin{cases}
								1 & \mbox{ if }n=1\\
								0 & \mbox{ if }p^{2}\mid n\mbox{ for some prime }p\\
								(k-1)^{\omega(n)}& \mbox{ if }n\mbox{ is square-free}
							\end{cases}
				\end{align*}
		\end{theorem}

		\begin{proof}
			If $n=1$, then $\delta_{k}(1)=k^{\omega(1)}=1$. We already know that $\delta$ is multiplicative, so we can focus on $\delta_{k}(p^{e})$ first. We have
				\begin{align*}
					k^{\omega(p)}
						& = \delta_{k}(1)+\delta_{k}(p)
				\end{align*}
			so $\delta_{k}(p)=k-1$. For $e>1$,
				\begin{align*}
					k^{\omega(p^{e})}
						& = \delta_{k}(1)+\delta_{k}(p)+\ldots+\delta_{k}(p^{e})
				\end{align*}
			Since $\omega(p^{e})=1$, we have that $\delta_{k}(p^{e})=0$ for $e>1$. Due to multiplicative nature of $\delta$, if $n=p^{e}s$ with $p\nmid s$ and $e>1$, then $\delta_{k}(n)=\delta_{k}(p^{e})\delta_{k}(s)=0$. Now, we only have to find $\delta_{k}(n)$ for square-free $n$. We will show by induction that $\delta_{k}(p_{1}\cdots p_{n})=(k-1)^{n}$. The case $r=1$ is already shown. So, assume that $\delta_{k}(p_{1}\cdots p_{r-1})(k-1)^{r-1}$. Then
				\begin{align*}
					k^{\omega(p_{1}\cdots p_{r})}
						& = \sum_{d\mid p_{1}\cdots p_{r}}\delta_{k}(d)
				\end{align*}
			If we choose a subset of $s$ primes from the set $\{p_{1},\ldots,p_{r}\}$ for $s<r$, then for any subset $\{p_{i_{1}},\ldots,p_{i_{s}}\}$,
				\begin{align*}
					\delta_{k}(p_{i_{1}}\cdots p_{i_{s}})
						& = (k-1)^{s}
				\end{align*}
			Since there are $\binom{r}{s}$ such subsets, we see that
				\begin{align*}
					k^{r}
						& = \binom{r}{0}\delta_{k}(1)+\binom{r}{1}\delta_{k}(p_{1})+\ldots+\binom{r}{r-1}\delta_{k}(p_{1}\cdots p_{r-1})+\delta_{k}(p_{1}\cdots p_{r})\\
					((k-1)+1)^{r}
						& = \binom{r}{0}+\binom{r}{1}(k-1)+\ldots+\binom{r}{r-1}(k-1)^{r-1}+\delta_{k}(p_{1}\cdots p_{r})
				\end{align*}
			From this, it is obvious that $\delta_{k}(p_{1}\cdots p_{r})=(k-1)^{r}$.
		\end{proof}
	We do not say $\delta$ is a direct generalization of $\mu$, rather it should be considered a generalization of $\mu^{2}$ which generalizes the well known identity
		\begin{align*}
			2^{\omega(n)}
				& = \sum_{d\mid n}\mu^{2}(d)
		\end{align*}
	Evidently, we have the following result using that $\lfloor x\rfloor=x+O(1)$.
		\begin{theorem}\label{thm:T}
			For positive integers $k$ and $n$,
				\begin{align*}
					\sum_{i=1}^{n}k^{\omega(n)}
						& = \sum_{i=1}^{n}\left\lfloor{\dfrac{n}{i}}\right\rfloor \delta_{k}(i)\\
						& = \sum_{i=1}^{n}\Delta\left(\dfrac{n}{i}\right)\\
						& = n\sum_{i=1}^{n}\dfrac{\delta_{k}(i)}{i}+O\left(\Delta_{k}(n)\right)
				\end{align*}
		\end{theorem}
	where
		\begin{align*}
			\Delta_{k}(x)
				& = \sum_{i\leq x}\delta_{k}(i)
		\end{align*}
	Clearly, \autoref{thm:T} tells us that $\mathcal{T}$ is a weighted sum of reciprocals of square-free numbers not exceeding $n$ and that the error term is the sum of $\delta_{k}(d)$ for all square-free $d\leq n$. Let $Q(x)$ be the number of square-free numbers not exceeding $x$. Then
		\begin{align*}
			Q(x)
				& = \sum_{n\leq x}\mu^{2}(n)\\
				& = \dfrac{6x}{\pi^{2}}+O(\sqrt{x})
		\end{align*}
	Since $\delta$ is a generalization of $\mu^{2}$, the author hopes that the study of $\Delta_{k}(x)$ will be of high interest to us because the error term of $Q(x)$ is not easily improved with elementary means. \textcite{liu_2016} currently has the best possible estimate $O(x^{11/35+\epsilon})$ for the error term of $Q$ \textit{assuming the Riemann Hypothesis}. In fact, all the good estimates\footnote{consult \textcite{liu_2016}} of this error term assumes the Riemann Hypothesis. We have seen in case of the prime number theorem, the study of weighted sums of Von Mangoldt function and Chebyshev functions eventually led to the elementary proof of the prime number theorem. Similarly, the author hopes that the study of $\Delta$ and the weighted sums of $\delta$ might shed some light on the distribution of square-free numbers. Let us try to calculate $\mathcal{T}$ as follows.
		\begin{align*}
			\mathcal{T}_{k}(x)
				& = \sum_{n\leq x}\left\lfloor{\dfrac{x}{n}}\right\rfloor \delta_{k}(n)\\
				& = \sum_{n\leq x}\left\lfloor{\dfrac{x}{n}}\right\rfloor(k-1)^{\omega(n)}-\sum_{\substack{n\leq x\\s^{2}\mid n\\s>1}}\left\lfloor{\dfrac{x}{n}}\right\rfloor(k-1)^{\omega(n)}\\
				& = \mathcal{A}_{k}(x)-\mathcal{B}_{k}(x)
		\end{align*}
	where
		\begin{align}
			\mathcal{A}_{k}(x)
				& = \sum_{n\leq x}\left\lfloor{\dfrac{x}{n}}\right\rfloor(k-1)^{\omega(n)}\label{eqn:A}\\
			\mathcal{B}_{k}(x)
				& = \sum_{\substack{n\leq x\\s^{2}\mid n\\s>1}}\left\lfloor{\dfrac{x}{n}}\right\rfloor(k-1)^{\omega(n)}\label{eqn:B}
		\end{align}
	At this point, we should define the functions $\mathcal{C},\mathcal{E}$ and $\mathcal{F}$.
		\begin{align}
			\mathcal{C}_{k}(x,m)
				& = \sum_{n\leq x}\left\lfloor{\dfrac{x}{n}}\right\rfloor(k-1)^{\omega(mn)}\label{eqn:C}\\
			\mathcal{E}_{k}(x, m)
				& =\sum_{\substack{n\leq x\\\gcd(m,n)=1}}\left\lfloor{\dfrac{x}{n}}\right\rfloor(k-1)^{\omega(n)}\label{eqn:E}\\
			\mathcal{F}_{k}(x, m)
				& = \sum_{\substack{n\leq x\\\gcd(m,n)>1}}\left\lfloor{\dfrac{x}{n}}\right\rfloor(k-1)^{\omega(mn)}\label{eqn:F}
			%\mathcal{G}_{k}(x, m)
			%	& = \sum_{n\leq x}\left\lfloor{\dfrac{x}{n}}\right\rfloor(k-1)^{\omega(mn)}\label{eqn:G}
		\end{align}
	We will show that we can write $\mathcal{T}$ in terms of $\mathcal{E}$ and $\mathcal{F}$. In order to do that, we require the auxiliary function $\mathcal{C}$.We can write $\mathcal{B}$ as
		\begin{align*}
			\mathcal{B}_{k}(x)
				& = \sum_{d\leq x}\mu(d)\sum_{r\leq\lfloor{x/d^{2}}\rfloor}\left\lfloor{\dfrac{x/d^{2}}{r}}\right\rfloor(k-1)^{\omega(d^{2}r)}\\
				& = \sum_{d\leq x}\mu(d)\mathcal{C}_{k}(\left\lfloor{x/d^{2}}\right\rfloor,d^{2})
		\end{align*}
	Note that $\omega(m^{2}n)=\omega(mn)$. So, $\mathcal{C}_{k}(x,d^{2})=\mathcal{C}_{k}(x,d)$. First, see the following for $\mathcal{E}$.
		\begin{align*}
			\mathcal{E}_{k}(x, m)
				& = \sum_{n\leq x}\left\lfloor{\dfrac{x}{n}}\right\rfloor(k-1)^{\omega(n)}-\sum_{\substack{n\leq x\\\gcd(m,n)>1}}\left\lfloor{\dfrac{x}{n}}\right\rfloor(k-1)^{\omega(n)}
		\end{align*}
	Now,
		\begin{align*}
			\sum_{\substack{n\leq x\\\gcd(m,n)>1}}\left\lfloor{\dfrac{x}{n}}\right\rfloor(k-1)^{\omega(mn)}
				& = \sum_{d\mid m}\mu(d)\sum_{n\leq\lfloor{x/d}\rfloor}\left\lfloor{\dfrac{x}{dn}}\right\rfloor(k-1)^{\omega(dn)}\\
				& = \sum_{d\mid m}\mu(d)\mathcal{C}_{k}(\lfloor{x/d}\rfloor,d)
		\end{align*}
	Thus, we get a recursive formula for $\mathcal{E}$.
		\begin{lemma}\label{lem:E}
			For positive integer $m$,
				\begin{align*}
					\mathcal{E}_{k}(x, m)
						& = \mathcal{A}_{k}(x)-\sum_{d\mid m}\mu(d)\mathcal{C}_{k}(\lfloor{x/d}\rfloor, d)
				\end{align*}
		\end{lemma}
	For $\mathcal{F}$, we have
		\begin{align*}
			\mathcal{F}_{k}(x, m)
				& = \sum_{\substack{n\leq x\\\gcd(m,n)>1}}\left\lfloor{\dfrac{x}{n}}\right\rfloor(k-1)^{\omega(mn)}\\
				& = \sum_{d\mid m}\mu(d)\sum_{di\leq x}\left\lfloor{\dfrac{x}{di}}\right\rfloor(k-1)^{\omega(mdn)}
		\end{align*}
	Since $d\mid m$, we have $\omega(mdn)=\omega(mn)$. Then,
		\begin{align*}
			\mathcal{F}_{k}(x, m)
				& = \sum_{d\mid m}\mu(d)\sum_{n\leq x/d}\left\lfloor{\dfrac{x/d}{n}}\right\rfloor(k-1)^{\omega(mn)}
		\end{align*}
	We get the following.
		\begin{lemma}\label{lem:F}
			For positive integer $m$,
				\begin{align*}
					\mathcal{F}_{k}(x, m)
						& = \sum_{d\mid m}\mu(d)\mathcal{C}_{k}(\lfloor{x/d}\rfloor, m)
				\end{align*}
		\end{lemma}
	Now, we are able to completely characterize $\mathcal{C}$ in terms of $\mathcal{E}$ and $\mathcal{F}$ which in turn implies that $\mathcal{B}$ can be characterized with $\mathcal{E}$ and $\mathcal{F}$ as well.
		\begin{theorem}\label{thm:C}
			For a positive integer $m$, $\mathcal{C}$ can be represented in terms of $\mathcal{E}$ and $\mathcal{F}$ as follows.
				\begin{align*}
					\mathcal{C}_{k}(x, m)
						& = (k-1)^{\omega(m)}\mathcal{E}_{k}(x, m)+\mathcal{F}_{k}(x, m)
				\end{align*}
		\end{theorem}

		\begin{proof}
			Write $\mathcal{C}$ as
				\begin{align*}
					\mathcal{C}_{k}(x, m)
						& = \sum_{\substack{n\leq x\\\gcd(m,n)=1}}\left\lfloor{\dfrac{x}{n}}\right\rfloor(k-1)^{\omega(mn)}+\sum_{\substack{n\leq x\\\gcd(m,n)>1}}\left\lfloor{\dfrac{x}{n}}\right\rfloor(k-1)^{\omega(mn)}\
				\end{align*}
			Since $\omega(mn)=\omega(m)+\omega(n)$ for $\gcd(m,n)=1$,
				\begin{align*}
					\mathcal{C}_{k}(x, m)
						& = (k-1)^{\omega(m)}\sum_{\substack{n\leq x\\\gcd(m,n)=1}}\left\lfloor{\dfrac{x}{n}}\right\rfloor(k-1)^{\omega(n)}+\mathcal{F}_{k}(x, m)\\
						& = (k-1)^{\omega(m)}\mathcal{E}_{k}(x, m)+\mathcal{F}_{k}(x, m)
				\end{align*}
			This proves the result.
		\end{proof}
	Finally, we express $\mathcal{A}$ in terms of $\mathcal{E}$ and $\mathcal{F}$.
		\begin{align*}
			\mathcal{A}_{k}(x)
				& = \sum_{n\leq x}\left\lfloor{\dfrac{x}{n}}\right\rfloor(k-1)^{\omega(n)}\\
				& = \sum_{s\leq x}\mu(s)\sum_{n\leq x/s}\left\lfloor{\dfrac{x}{ns}}\right\rfloor(k-1)^{\omega(ns)}\\
				& = \sum_{s\leq x}\mu(s)\left(\sum_{\substack{n\leq x/s\\\gcd(n,s)=1}}\left\lfloor{\dfrac{x}{ns}}\right\rfloor(k-1)^{\omega(ns)}+\sum_{\substack{n\leq x/s\\\gcd(n,s)>1}}\left\lfloor{\dfrac{x}{ns}}\right\rfloor(k-1)^{\omega(ns)}\right)\\
				& = \sum_{s\leq x}\mu(s)\left((k-1)^{\omega(s)}\sum_{\substack{n\leq x/s\\\gcd(n,s)=1}}\left\lfloor{\dfrac{x/s}{n}}\right\rfloor(k-1)^{\omega(n)}+\sum_{\substack{n\leq x/s\\\gcd(n,s)>1}}\left\lfloor{\dfrac{x/s}{n}}\right\rfloor(k-1)^{\omega(ns)}\right)
		\end{align*}
	Thus, we have the following result.
		\begin{theorem}\label{thm:A}
			$\mathcal{A}$ can be represented in terms of $\mathcal{E}$ and $\mathcal{F}$ as follows.
				\begin{align*}
					\mathcal{A}_{k}(x)
						& = \sum_{s\leq x}\mu(s)\left((k-1)^{\omega(s)}\mathcal{E}_{k}(\lfloor{x/s}\rfloor, s)+\mathcal{F}_{k}(\lfloor{x/s}\rfloor, s)\right)
				\end{align*}
		\end{theorem}
	By \autoref{thm:A} and \autoref{thm:C}, we have now expressed both $\mathcal{A}$ and $\mathcal{B}$, hence also $\mathcal{T}$ in terms of $\mathcal{F}$ and $\mathcal{F}$ which are easier to calculate. Thus, we have the following.
		\begin{theorem}\label{thm:T2}
			$\mathcal{T}$ can be represented in terms of $\mathcal{E}$ and $\mathcal{F}$ as follows.
				\begin{align*}
					\mathcal{T}_{k}(x)
						%& = \sum_{s\leq x}\mu(s)\left((k-1)^{\omega(s)}\mathcal{E}_{k}(\lfloor{x/s}\rfloor, s)+\mathcal{F}_{k}(\lfloor{x/s}\rfloor, s)\right)-\sum_{d\leq x}\mu(d)\mathcal{C}_{k}(\left\lfloor{x/d^{2}}\right\rfloor,d)\\
						& = \sum_{s\leq x}\mu(s)\left((k-1)^{\omega(s)}\mathcal{E}_{k}(\lfloor{x/s}\rfloor, s)+\mathcal{F}_{k}(\lfloor{x/s}\rfloor, s)\right)\\
						& -\sum_{d\leq x}\mu(d)\left((k-1)^{\omega(d)}\mathcal{E}_{k}(\lfloor{x/d^{2}}\rfloor, d)+F_{k}(\lfloor{x/d^{2}}\rfloor,d)\right)
				\end{align*}
		\end{theorem}
	An immediate application of \autoref{thm:T2} is the following result for $k=2$. But we will skip using $\mathcal{E}$ and $\mathcal{F}$ since $\delta_{2}(n)=1$ for all square-free $n$.
		\begin{theorem}
			For a positive integer $n$,
				\begin{align}
					\sum_{i=1}^{n}2^{\omega(i)}
						& = \sum_{1\leq i^{2}\leq n}\mu(i)D_{2}\left(\left\lfloor{\dfrac{n}{i^{2}}}\right\rfloor\right)\label{eqn:T2}\\
						& = 2n-1+2\sum_{1<i^{2}<n}\left(\varphi\left(\left\lfloor{\dfrac{n}{i}}\right\rfloor,i\right)-\varphi(i)\right)\label{eqn:T23}
				\end{align}
			where $\varphi(x, a)$ is the number of positive integers not exceeding $x$ which are relatively prime to $a$ and $\varphi(n)=\varphi(n,n)$.
		\end{theorem}

		\begin{proof}
			We have that $\delta_{2}(n)=1$ if $n$ is square-free, otherwise $0$.
				\begin{align*}
					\mathcal{T}_{2}(n)
						& = \sum_{i=1}^{n}\left\lfloor{\dfrac{n}{i}}\right\rfloor \delta_{2}(i)\\
						& = \sum_{i=1}^{n}\left\lfloor{\dfrac{n}{i}}\right\rfloor-\sum_{\substack{1\leq i\leq n\\d^{2}\mid i\\d>1}}\left\lfloor{\dfrac{n}{i}}\right\rfloor
				\end{align*}
			Now, for $1<s^{2}<n$, $s^{2}i$ is not square-free for $1\leq i\leq\lfloor{n/s^{2}}\rfloor$. Now, using principle of inclusion of exclusion, we easily see that \autoref{eqn:T2} holds.

			For proving \autoref{eqn:T23}, we consider the definition of $\mathcal{T}_{2}(n)$. By definition, $\mathcal{T}_{2}(n)$ is the number of ordered solutions $(a,b)$ to the inequality $ab\leq n$ where $\gcd(a,b)=1$. Since $(a,1)$ and $(1,b)$ are solutions for all $1\leq a,b\leq n$, there are $2n-1$ such pairs (the pair $(1,1)$ is common in both). If $(a,b)$ is a solution, then so is $(b,a)$. So $a,b>1$ and we can assume without loss of generality that $a<b$ because $(a,a)$ is not a solution since $\gcd(a,b)=a>1$. We have $1<a<b\leq\lfloor{n/a}\rfloor$ and that $\gcd(a,b)=1$. For a fixed $a$, the number of such $b$ is
				\begin{align*}
					2(\varphi(\lfloor{n/a}\rfloor)-\varphi(a))
				\end{align*}
			since $(b,a)$ is a different solution that $(a,b)$. Also, $a^{2}<ab\leq n$ so $1<a^{2}<n$. Thus, the total number of solutions is
				\begin{align*}
					2n-1+2\sum_{1<a^{2}<n}\left(\varphi\left(\left\lfloor\dfrac{n}{a}\right\rfloor,a\right)-\varphi(a)\right)
				\end{align*}
		\end{proof}

		\begin{remark}
			It should be noted that \autoref{lem:E}, \autoref{lem:F}, \autoref{thm:C}, \autoref{thm:A} and \autoref{thm:T2} should apply in general to a certain class of functions $F$ of the form
				\begin{align*}
					F(x)
						& = \sum_{n\leq x}a(n)f(n)
				\end{align*}
			where $a$ has the same property as
				\begin{align*}
					\left\lfloor{\dfrac{x}{mn}}\right\rfloor
						& = \left\lfloor{\dfrac{\left\lfloor{x/m}\right\rfloor}{n}}\right\rfloor
				\end{align*}
			and $f$ is a multiplicative function such that $f(dmn)=f(mn)$ if $d\mid n$ or $d\mid m$.
		\end{remark}
	\printbibliography
\end{document}